\documentclass[10pt]{article}
\usepackage{epsf,amsmath,amscd,amsthm,amsfonts,latexsym,color}
\usepackage{graphicx}
\usepackage{amssymb}
\usepackage{epstopdf}

\theoremstyle{plain}
\newtheorem{theorem}                 {Theorem}      [section]
\newtheorem{teo}                 {Theorem}      [section]
\newtheorem{prop}  [theorem]  {Proposition}
\newtheorem{corollary}    [theorem]  {Corollary}
\newtheorem{defi}{Definition}
\newtheorem{lem}        [theorem]  {Lemma}

\theoremstyle{definition}

\newtheorem*{rmk*} {Remark}

\def \r{\mbox{${\mathbb R}$}}
\def \C{\mbox{${\mathbb C}$}}
\def \H{\mbox{${\mathbb H}$}}
\def \h{\mbox{${\mathfrak h}$}}
\def \z{\mbox{${\mathbb Z}$}}
\def \g{\mathfrak{g}}
\begin{document}

\title{ Pseudo-Riemannian Symmetries on Heisenberg group $\mathbb{H}_{3}$}
\vspace{2 cm}
\author{Michel Goze \thanks{The first author was supported by: Visiting professor program, Regione Autonoma della Sardegna - Italy. The second author was supported by GNSAGA(Italy)  } \ ,  Paola Piu}
%\date{}

\maketitle
\begin{abstract}
The notion of  $\Gamma$-symmetric space is a natural generalization of the classical notion of  symmetric space based on $\z_2$-grading of Lie algebras. In our case, we consider homogeneous spaces $G/H$ such that the Lie algebra $\g$ of $G$ admits a $\Gamma$-grading where $\Gamma$ is a finite abelian group.
In this work we study Riemannian metrics and Lorentzian metrics on the Heisenberg group $\mathbb{H}_3$ adapted to the symmetries of a $\Gamma$-symmetric structure on $\mathbb{H}_3$. We prove that the classification of $\z_2^2$-symmetric Riemannian and Lorentzian metrics on $\mathbb{H}_3$ corresponds to the classification  of left invariant Riemannian and Lorentzian metrics, up to isometries. This gives examples of non-symmetric Lorentzian homogeneous spaces.

\end{abstract}

\vskip2mm
\noindent{\bf Mathematics Subject Classification 2000:}
22F30, 53C30, 53C35,17B70
{\bf Key words: }$\Gamma$-symmetric spaces, Heisenberg group, Graded Lie algebras.

\section{$\Gamma $-symmetric spaces}
Let $\Gamma $ be a finite abelian group. A $\Gamma$-{\it symmetric space} is an homogeneous space $ G/H $ such that there exists an injective homomorphism
\[
\rho : \Gamma \to Aut(G)
\]
 where $ Aut (G) $ is the group of automorphisms of the Lie group $ G $, the subgroup $ H $ satisfying $G^{\Gamma}_{e} \subset H \subset G^{\Gamma}$  where
$ G^{\Gamma} = \left\{ x\in G/ \rho(\gamma)(x) = x, \forall \gamma \in \Gamma \right\}$ and $G^{\Gamma}_{e}$ is the connected identity component of $ G^{\Gamma} $  of $G$.\\
The notion of $\Gamma $-symmetric space is a generalization of the classical notion of symmetric space by considering a general finite abelian group of symmetries $\Gamma$ instead of  $\z_2$. The case
 $\Gamma = \z_k$, the cyclic group of order $ k $, was considered by
 A.J. Ledger, M. Obata \cite{L.O}, A. Gray, J. A. Wolf,  \cite{[G.W]} and O. Kowalski \cite{[K]} in terms of $k$-symmetric spaces. The general notion of $\Gamma$-symmetric spaces was introduced by R. Lutz \cite{[L]} and was algebraically reconsidered by Y. Bahturin and M. Goze \cite{[B.G]}.
In this last work the authors proved, in particular, that a  $\Gamma$-symmetric space $M=G/H$ is reductive and the Lie algebra $\mathfrak g$ of $G$ is  $\Gamma$-graded, that is,
\[
\mathfrak{g} =  \underset{\gamma \in \Gamma}{\oplus}\; \mathfrak{g}_{\gamma}
\]
with
\[
[\mathfrak{g}_{\gamma},\mathfrak{g}_{\gamma'}] \subset \mathfrak{g}_{\gamma \gamma'}
\qquad \qquad  \forall \; \gamma, \gamma' \in \Gamma.
\]
\noindent{\bf Examples.}
\begin{enumerate}
\item If $\Gamma = \z_2$ and $\g$  a  complex or real Lie algebra, a $\Gamma$-grading of $\g$ corresponds to the classical symmetric decomposition of $\g$.
\item If $\g$ is a simple complex Lie algebra and $\Gamma=\z_k$, $k \geq 3$, we have the notion of generalized symmetric spaces and the classification of $\Gamma$-gradings are described by V. Kac in \cite{[Ka]}.
\item Let $\mathfrak {g} = \oplus _ {\gamma \in \Gamma} \mathfrak {g}_ {\gamma}$ be a Lie algebra $ \Gamma $-graded. For any commutative associative algebra $\mathcal {A}$, the
current algebra $ \mathcal {A} \otimes\mathfrak {g} $ (see \cite{[G.R prod tens]}) also admits a $\Gamma $-grading.
\item In \cite{[B.G]}, the $\z_2 \times \z_2$ grading on classical simple complex Lie algebras are classified.

    \end{enumerate}

\medskip

One proves also in \cite{[B.G]} that the structure of $\Gamma$-symmetric space on $G/H$ is, when $G$ is connected, completely determinate by the $\Gamma$-grading of $\g$. Thus, if $G$ is connected, the classification of the $\Gamma$-symmetric spaces is equivalent to the classification of the $\Gamma$-graded Lie algebras. Many results of this last problem concern more particularly  the simple Lie algebras. For solvable or nilpotent Lie algebras, it is an open problem. A first approach is to study induced grading on Borel or parabolic subalgebras of simple Lie algebras. In this work we describe $\Gamma$-grading of the Heisenberg algebra $\h_3$. Two reasons for this study
\begin{itemize}
\item Heisenberg algebras are nilradical of some Borel subalgebras.
\item The Riemannian and Lorentzian geometries on the $3$-dimensional Heisenberg group have been studied recently by many authors. Thus it is interesting to study the Riemannian and Lorentzian symmetries with the natural symmetries associated with a $\Gamma$-symmetric structure on the Heisenberg group. In this paper we prove that these geometries are entirely determinate by Riemannian and Lorentzian structures adapted to $(\z^2 \times \z^2)$-symmetric structures.
\end{itemize}

Recall that the Heisenberg algebra $\mathfrak{h}_3$ is the real Lie algebra whose elements are matrices
\[
\begin{pmatrix}
0&x&z\\
0&0&y\\
0&0&0
\end{pmatrix}
\qquad  \mbox{with} \qquad x,y,z \in \r
\]
The elements of  $\mathfrak h_{3}$, $X_1$, $X_2$, $X_3$, corresponding to $(x,y,z) =(1,0,0),(0,1,0)$ and $(0,0,1)$ form a basis of $\mathfrak h_{3}$ and the Lie brackets are given in this basis by
\[
\begin{cases}
[X_1,X_2] = X_3\\
[X_1,X_3] = [X_2,X_3] = 0 .
\end{cases}
\]
The Heisenberg group is the real Lie group of dimension $ 3 $ consisting of matrices
\[
\begin{pmatrix}
1&a&c\\
0&1&b\\
0&0&1
\end{pmatrix}
\qquad \qquad a,b,c \in \r
\]
and its Lie algebra is $\mathfrak h_3$.
\section{Finite abelian subgroups of  $Aut(\mathfrak h_3)$}
Let us denote by $Aut(\h_3)$ the group of automorphisms of the Heisenberg algebra $\h_3$.
Every  $\tau \in Aut(\h_3)$ admits, with regards to the basis $\{X_1,X_2,X_3\}$, the following matricial representation:
\[
\begin{pmatrix}
\alpha_1&\alpha_2&0\\
\alpha_3&\alpha_4&0\\
\alpha_5&\alpha_6&\Delta
\end{pmatrix}
\]
with $\Delta = \alpha_1 \alpha_4 - \alpha_2 \alpha_3 \neq 0$.

Let $\Gamma $ be a finite  abelian subgroup of $Aut(\h_3)$. It admits a cyclic decomposition. If $\Gamma $ contains a cyclical component isomorphic to $\z_k $,  then there exists an automorphism $\tau $ satisfying $\tau^ k = Id $. The aim of this section is to determinate  the cyclic decomposition of any finite abelian subgroup $\Gamma$.

\subsection{Subgroups of $Aut(\h_3)$ isomorphic to $\z_2$}

Let $\tau \in Aut(\h_3)$ satisfying $\tau^2 = Id $. If
\[
\begin{pmatrix}
\alpha_1&\alpha_2&0\\
\alpha_3&\alpha_4&0\\
\alpha_5&\alpha_6&\Delta
\end{pmatrix}
\]
is its matricial representation, then the involution can be written in matrix form
\[
\begin{pmatrix}
\alpha_1^{2}+ \alpha_{2}\alpha_{3}&\alpha_{1}\alpha_2 + \alpha_{2}\alpha_{4}&0\\
\alpha_1\alpha_3+\alpha_3 \alpha_4&\alpha_{2}\alpha_{3}+\alpha_4^{2}&0\\
\alpha_1\alpha_5+ \alpha_3 \alpha_6 + \Delta \alpha_5&
\alpha_2\alpha_5+\alpha_4\alpha_6+\Delta \alpha_6&\Delta^{2}
\end{pmatrix} =
\begin{pmatrix}
1&0&0\\
0&1&0\\
0&0&1
\end{pmatrix}
\]
The resolution of the system can be done using formal calculation software. Here we use Mathematica.
\begin{prop}\label{lemma1}
Any involutive automorphism $\tau $ of $ Aut (\mathfrak h_3) $ is equal to one of the following automorphisms
\[
Id, \ \tau_1(\alpha_3,\alpha_6) = \begin{pmatrix}
-1&0&0\\
\alpha_3&1&0\\
\displaystyle\frac{\alpha_{3}\alpha_6}{2}&\alpha_6&-1
\end{pmatrix},  \tau_2(\alpha_3,\alpha_5) = \begin{pmatrix}
1&0&0\\
\alpha_3&-1&0\\
\alpha_5&0&-1
\end{pmatrix},
\]
\[
\tau_3(\alpha_1,\alpha_{2}\neq 0,\alpha_6) = \begin{pmatrix}
\alpha_1&\alpha_{2}&0\\
\displaystyle\frac{1-\alpha_1^{2}}{\alpha_{2}}&-\alpha_{1}&0\\
\displaystyle\frac{(1 + \alpha_{1})\alpha_6}{\alpha_{2}}&\alpha_6&-1
\end{pmatrix},
 \tau_4(\alpha_5,\alpha_6) = \begin{pmatrix}
-1&0&0\\
0&-1&0\\
\alpha_5&\alpha_{6}&1
\end{pmatrix}  .
\]
\end{prop}
\begin{corollary}
Any subgroup of $Aut(\h_3)$ isomorphic to $\z_{2}$ is equal to one of the following:
\begin{enumerate}
\item $\Gamma_1(\alpha_3,\alpha_6)=\{Id,\tau_1(\alpha_3,\alpha_6)\}$,
\item $\Gamma_2(\alpha_3,\alpha_5)=\{Id,\tau_2(\alpha_3,\alpha_5)\}$,
\item $\Gamma_3(\alpha_1,\alpha_2,\alpha_6)=\{Id,\tau_3(\alpha_1,\alpha_2,\alpha_6), \ \alpha_2 \neq 0\}$,
\item $\Gamma_4(\alpha_5,\alpha_6)=\{Id,\tau_4(\alpha_5,\alpha_6)\}$.
\end{enumerate}
\end{corollary}

\subsection{Subgroups of $Aut(\h_3)$ isomorphic to $\z_3$}
Let $\tau$ be an automorphism  satisfying $\tau^3 = Id$. This identity is equivalent to
$
\tau^{2} = \tau^{-1}.
$
If we have
\[
\tau = \begin{pmatrix}
\alpha_1&\alpha_2&0\\
\alpha_3&\alpha_4&0\\
\alpha_5&\alpha_6&\Delta
\end{pmatrix}
,\qquad \qquad \Delta = \alpha_{1}\alpha_{4} - \alpha_{2}\alpha_{3},
\]
 then
\[
\tau^{-1}= \frac{1}{\Delta}
\begin{pmatrix}
\alpha_4&-\alpha_2&0\\
-\alpha_3&\alpha_1&0\\
\displaystyle\frac{\alpha_{3}\alpha_{6} -\alpha_{4}\alpha_5}{\Delta}&\displaystyle\frac{\alpha_{2}\alpha_5-\alpha_{1}\alpha_6}{\Delta}&1
\end{pmatrix}
\]
The condition $\tau^{2} = \tau^{-1} $ implies
$
\Delta^{3} = 1 
$ and the only real solution is $\Delta =  1.$ 
Thus $\tau^{2} = \tau^{-1}$ is equivalent to
\begin{equation}\label{sist-t3}
\begin{cases}
\alpha_{1}\alpha_{4} - \alpha_{2}\alpha_{3} =1,\\
\alpha_{1}^{2} +\alpha_{2} \alpha_{3} = \alpha_{4},\\
\alpha_{4}^{2} +\alpha_{2} \alpha_{3} = \alpha_{1}, \\
\alpha_{2}\left(1 + \alpha_{1} + \alpha_{4}\right) = 0,\\
\alpha_{3}\left(1 + \alpha_{1} + \alpha_{4}\right) = 0,\\
\alpha_{5}\left(1 + \alpha_{1} + \alpha_{4}\right) = 0,\\
\alpha_{6}\left(1 + \alpha_{1} + \alpha_{4}\right) = 0.
\end{cases}
\end{equation}
If  $\alpha_{1} + \alpha_{4} \neq -1,$  then
$
\alpha_{2} = \alpha_{3} =  \alpha_{5} = \alpha_{6} = 0$ and $ \alpha_{1} = \alpha_{4} = 1.
$
In this case
\[
\tau = Id.
\]
Let us assume $\alpha_{1} + \alpha_{4} = -1$. Then  \eqref{sist-t3}
is reduced to
\[
\alpha_{1}^{2} +\alpha_{1} + \alpha_{2} \alpha_{3} + 1 = 0
\]
If $\alpha_{2} \alpha_{3} > - \frac{3}{4},$ we have no solutions.
Assume  that $\alpha_{2} \alpha_{3} \leq - \frac{3}{4}$.  Then
\[
\alpha_{1} = \frac{-1 \pm \sqrt{-3 - 4 \alpha_{2} \alpha_{3}}}{2}.
\]
So we obtain
\[
\tau_{5} = \begin{pmatrix}
\displaystyle\frac{-1 - \sqrt{-3 - 4 \alpha_{2} \alpha_{3}}}{2}&\alpha_2&0\\
\alpha_3&\displaystyle\frac{-1 + \displaystyle\sqrt{-3 - 4 \alpha_{2} \alpha_{3}}}{2}&0\\
\alpha_5&\alpha_6&1
\end{pmatrix},
\]
and
\[
\tau'_{5} = \begin{pmatrix}
\displaystyle\frac{-1 + \sqrt{-3 - 4 \alpha_{2} \alpha_{3}}}{2}&\alpha_2&0\\
\alpha_3&\displaystyle\frac{-1 - \sqrt{-3 - 4 \alpha_{2} \alpha_{3}}}{2}&0\\
\alpha_5&\alpha_6&1
\end{pmatrix}.
\]
Since
\[
\tau^{2}_{5}(\alpha_2, \alpha_3, \alpha_5, \alpha_6) = \tau'_{5}(-\alpha_{2}, -\alpha_{3}, \alpha'_{5},\alpha'_{6})
\]
where
$$
\alpha'_{5}=\displaystyle\frac{\alpha_5 - \sqrt{-3 - 4 \alpha_2 \alpha_3} \alpha_5 - 2 \alpha_3 \alpha_6}{2}, \ \ 
\alpha'_{6}= \displaystyle\frac{\alpha_6 + \sqrt{-3 - 4 \alpha_2 \alpha_3} \alpha_6 -2 \alpha_2 \alpha_5 }{2}
$$
and
\[
\tau'^{2}_{5}(\alpha_2, \alpha_3, \alpha_5, \alpha_6) = \tau_{5}(-\alpha_{2}, -\alpha_{3}, \alpha''_{5},\alpha''_{6})
\]
where
$$
\alpha''_{5}=\displaystyle\frac{\alpha_5 + \sqrt{-3 - 4 \alpha_2 \alpha_3} \alpha_5 - 2 \alpha_3 \alpha_6}{2}, \ \
\alpha''_{6}= \displaystyle\frac{\alpha_6 - \sqrt{-3 - 4 \alpha_2 \alpha_3} \alpha_6 -2 \alpha_2 \alpha_5 }{2},
$$
we deduce

\begin{prop}
Any abelian subgroup of $Aut(\mathfrak h_3) $ isomorphic to $\z_3$ is equal to
\[
\Gamma_{5}(\alpha_2,\alpha_3,\alpha_5,\alpha_6) =\left\{Id,\tau_{5}(\alpha_2, \alpha_3, \alpha_5, \alpha_6),\tau'_{5}(-\alpha_2, -\alpha_3, \alpha'_5 , \alpha'_6), \ \ 4\alpha_{2} \alpha_{3} \leq -3 \right\}.
\]
\end{prop}

\subsection{Subgroups of $Aut(\h_3)$ isomorphic to $\z_k, \ k >3$}
If $\tau \in Aut(\h_3)$ satisfies $\tau^k=Id, \ k>3$, its minimal polynomial has $3
$ simple roots and it is of degree $3$. More precisely, it is written
\[
m_\tau(x)=(x-1)(x-\mu_k)(x-\overline{\mu_k})
\]
where $\mu_k$ is a root of order $k$ of $1$. As  $\tau$ has to generate a cyclic subgroup of $Aut(\h_3)$ isomorphic to $\z_k$, the root $\mu_k$ is a primitive root of $1$. There exists $m$, a prime number with $k$ such that $\mu_k= exp(\displaystyle\frac{2mi\pi}{k}).$ If
\[
\tau = \begin{pmatrix}
\alpha_1&\alpha_2&0\\
\alpha_3&\alpha_4&0\\
\alpha_5&\alpha_6&\Delta
\end{pmatrix}
\]
is the matricial representation of $\tau$, then $\Delta=1$ and $\alpha_1+\alpha_4=2\displaystyle\cos\frac{2m\pi}{k}$. Thus
\[
\left\{
\begin{array}{l}
\alpha_1=\displaystyle\cos\frac{2m\pi}{k} - \sqrt{\cos^2\frac{2m\pi}{k} - 1- \alpha_{2} \alpha_{3}},\\
\alpha_4=\displaystyle\cos\frac{2m\pi}{k} + \sqrt{\cos^2\frac{2m\pi}{k} - 1- \alpha_{2} \alpha_{3}},\\
\end{array}
\right.
\]
or
\[
\left\{
\begin{array}{l}
\alpha_1=\displaystyle\cos\frac{2m\pi}{k} + \sqrt{\cos^2\frac{2m\pi}{k} - 1- \alpha_{2} \alpha_{3}},\\
\alpha_4=\displaystyle\cos\frac{2m\pi}{k} - \sqrt{\cos^2\frac{2m\pi}{k} - 1- \alpha_{2} \alpha_{3}}.\\
\end{array}
\right.
\]
If $\tau'$ and $\tau ''$ denote the automorphisms corresponding to these solutions, we have, for a good choice of the parameters $\alpha_{i}$,  $\tau' \circ \tau '' =Id$ and $\tau''=(\tau')^{k-1}.$
Thus these automorphisms generate the same subgroup of $Aut(\h_3)$. Moreover, with same considerations, we can choose
$m=1$. Thus we have determinate the automorphism
$\tau_6(\alpha_2,\alpha_3,\alpha_5,\alpha_6)$ whose matrix is
\[
 \begin{pmatrix}
\displaystyle\cos\frac{2\pi}{k} + \sqrt{\cos^2\frac{2\pi}{k} - 1- \alpha_{2} \alpha_{3}}&\alpha_2&0\\
\alpha_3&\displaystyle\cos\frac{2\pi}{k} - \sqrt{\cos^2\frac{2\pi}{k} - 1- \alpha_{2} \alpha_{3}}&0\\
\alpha_5&\alpha_6&1
\end{pmatrix}
\]
\begin{prop}
Any abelian subgroup of $Aut(\mathfrak h_3) $ isomorphic to $\z_k$ is equal to
\[
\Gamma_{6,k}(\alpha_2,\alpha_3,\alpha_5,\alpha_6) =\left\{Id,\tau_{6}(\alpha_2, \alpha_3, \alpha_5, \alpha_6),\cdots,\tau_{6}^{k-1}, \ \ \alpha_{2} \alpha_{3} \leq -1+ \displaystyle\cos^2\frac{2\pi}{k}\right\}.
\]
\end{prop}

\subsection{General case}
Now suppose that the cyclic decomposition of a finite abelian subgroup $\Gamma$ of $Aut(\h_3)$ is isomorphic to
$\z_2^{k_2} \times \z_3^{k_3} \times \cdots \times \z_p^{k_p}$ with $k_i \geq 0.$ 
\begin{lem}
Let $\Gamma$ be an abelian finite subgroup of $Aut(\h_3)$ with a cyclic decomposition isomorphic to
\[
\z_2^{k_2} \times \z_3^{k_3} \times \cdots \times \z_p^{k_p}.
\]
Then
\begin{itemize}
\item If there is $i\geq 3$ such that $k_i \neq 0$, then $k_2 \leq 1.$
\item If $k_2 \geq 2$, then $\Gamma$ is isomorphic to $\z_2^{k_2}$.
\end{itemize}
\end{lem}
\begin{proof}
Assume that there is $i \geq 3$ such that $k_i \geq 1$.  If $k_2 \geq 1$, there exist two automorphisms $\tau$ and $\tau'$ satisfying $\tau'^{2}=\tau^{2}=Id$ and $\tau' \circ \tau=\tau\circ \tau'.$ Thus $\tau'$ and $\tau$ can be reduced simultaneously in the diagonal form and admit a common basis of eigenvectors. As for any $\sigma\in Aut(\h_3)$ we have $\sigma(X_3)=\Delta X_3$, $X_3$ is an eigenvector for $\tau'$ and $\tau$ associated to the eigenvalue $1$ for $\tau'$ and $\pm 1$ for $\tau$. As the two other eigenvalues of $\tau'$ are complex conjugate numbers, the corresponding eigenvectors are complex conjugate. This implies that the eigenvalues of $\tau$ distinguished of $\Delta=\pm 1$ are equal and  from Proposition \ref{lemma1}, $\tau=\tau_4(\alpha_5,\alpha_6)$. But
\[
\tau_4(\alpha_5,\alpha_6)\circ \tau_4(\alpha'_5,\alpha'_6)=\tau_4(\alpha'_5,\alpha'_6)\circ \tau_4(\alpha_5,\alpha_6) \Leftrightarrow \alpha_5=\alpha'_5, \ \alpha_6=\alpha'_6.
\]
\end{proof}

Thus, we have to determine, in a first step, the subgroups $\Gamma $ of $Aut (\mathfrak h_3)$ isomorphic a $(\z_2)^k $ with $k \geq 2$.
\begin{itemize}
\item Any involutive automorphism $\tau$ commuting with $\tau_1 (\alpha_3, \alpha_6) $ and distinct from it is equal to one of the following automorphisms
\[
\tau_2(-\alpha_3,\alpha_5) \qquad \qquad \tau_4(\alpha_5,-\alpha_6)
\]
Indeed, if we set $[\tau,\tau'] = \tau \circ\tau' -\tau' \circ\tau$  then
\begin{eqnarray*}
\left[\tau_1(\alpha_3,\alpha_6),\tau_1(\alpha'_3,\alpha'_6)\right] &=& 0 \qquad \mbox{if and only if} \qquad
\alpha'_3 =\alpha_3 \; \mbox{and} \; \alpha'_6 = \alpha_6
\\
\left[\tau_1(\alpha_3,\alpha_6),\tau_{2}(\alpha'_{3},\alpha'_{5})\right]&=& 0 \qquad \mbox{if and only if} \qquad
\alpha'_3 = - \alpha_3\\
\left[\tau_1(\alpha_3,\alpha_6),\tau_{3}(\alpha_{1},\alpha_{2},\alpha'_{3})\right]&\neq& 0 \qquad \mbox{whatever they are}\qquad  \alpha_{1},\alpha_{2},\alpha'_{3}\\
\left[\tau_1(\alpha_3,\alpha_6),\tau_4(\alpha_5,\alpha'_6)\right] &=& 0 \qquad \mbox{if and only if} \qquad
\alpha'_6 =-\alpha_6
\end{eqnarray*}
These results follow directly from the matrix calculation. In addition we have
 \[
\tau_1(\alpha_3,\alpha_6) \circ \tau_2(-\alpha_3,\alpha_5) = \tau_{4}\left(-\frac{\alpha_{3}\alpha_{6}}{2}-\alpha_{5},-\alpha_{6}\right)
\]
and
\[
\left[\tau_2(-\alpha_3,\alpha_5),\tau_{4}\left(-\frac{\alpha_{3}\alpha_{6}}{2}-\alpha_{5},-\alpha_{6}\right)\right] = 0.
\]
Thus
\[
\Gamma_{7}(\alpha_3,\alpha_5,\alpha_6) = \left\{Id,\tau_1(\alpha_3,\alpha_6),\tau_2(-\alpha_3,\alpha_5),\tau_{4}\left(-\frac{\alpha_{3}\alpha_{6}}{2}-\alpha_{5},-\alpha_{6}\right)\right\}
\]
is a subgroup of $Aut (\mathfrak h_ {3})$ isomorphic to $\z^{2}_{2} $. Moreover it is the only subgroup of $ Aut (\mathfrak h_3) $ of type $ (\z_2)^k $, $ k \geq 2 $, containing an automorphism of type $\tau_1(\alpha_3,\alpha_6)$.
\item Let us suppose that $ \tau_2(\alpha_3,\alpha_5) \in \Gamma$ and that $ \tau_1(\alpha_3,\alpha_6) \notin \Gamma$. We have
\begin{eqnarray*}
\left[\tau_2(\alpha_3,\alpha_5),\tau_2(\alpha'_3,\alpha'_5)\right] &=& 0 \qquad \mbox{if and only if} \qquad
\alpha'_3 =\alpha_3 \; \mbox{and} \; \alpha'_5 = \alpha_5
\\
\left[\tau_2(\alpha_3,\alpha_5),\tau_{3}(\alpha_{1},\alpha_{2},\alpha_{6})\right]&\neq& 0 \qquad \mbox{because by assumption} \qquad
\alpha_2 \neq 0\\
\left[\tau_2(\alpha_3,\alpha_5),\tau_4(\alpha'_5,\alpha_6)\right]&=& 0 \qquad \mbox{if and only if}\qquad  \alpha'_{5}= -\alpha_{5}-\frac{\alpha_{3}\alpha_{6}}{2}
\end{eqnarray*}
But
\[
\tau_2(\alpha_3,\alpha_5) \circ \tau_{4}(-\alpha_{5}-\frac{\alpha_{3}\alpha_{6}}{2},\alpha_{6}) =  \tau_1(\alpha_3,\alpha_6).
\]
Thus every abelian subgroup $\Gamma $ containing $\tau_2 (\alpha_3 \alpha_5)$ are either isomorphic to $\z_2 $ or is equal to $\Gamma_7$
\item Assume that $ \tau_3(\alpha_{1},\alpha_3,\alpha_6) \in \Gamma$. We have
\[
\left[\tau_3(\alpha_{1},\alpha_2,\alpha_6),\tau_3(\alpha'_{1},\alpha'_2,\alpha'_6)\right] = 0 \quad \mbox{if and only if} \quad \alpha'_{1}=-\alpha_{1} \; \mbox{and} \; \alpha'_{2}=-\alpha_{2}.
\]
Thus
\begin{eqnarray*}
\left[\tau_3(\alpha_{1},\alpha_2,\alpha_6),\tau_3(-\alpha_{1},-\alpha_2,\alpha'_6)\right] &=& 0,
\end{eqnarray*}
and
\begin{eqnarray*}
\left[\tau_3(\alpha_{1},\alpha_2,\alpha_6),\tau_4(\alpha_5,\alpha'_6)\right]&=& 0 \quad \mbox{if and only if}\quad  \alpha_{2}\alpha_{5}+ 2 \alpha_{6}= ( \alpha_{1}-1)\alpha'_{6}.
\end{eqnarray*}
Moreover
\[
\tau_3(\alpha_{1},\alpha_2,\alpha_6) \circ \tau_3(-\alpha_{1},-\alpha_2,\alpha'_6) =\tau_4\left(\frac{\alpha'_6(1 - \alpha_{1})-\alpha_{6}(1 + \alpha_{1})}{\alpha_{2}},-\alpha_{6}-\alpha'_6\right)
\]
because
\[
\alpha_{2}\left(\frac{\alpha'_6(1 - \alpha_{1})-\alpha_{6}(1 +\alpha_{1})}{\alpha_{2}}\right) + 2 \alpha_{6} + (1 - \alpha_{1})(-\alpha_{6}-\alpha'_6) = 0.
\]
The subgroup of  $\Gamma_8(\alpha_1,\alpha_2,\alpha_6,\alpha'_6)$ of $Aut(\mathfrak h_{3})$ equal to
\[
\left\{Id,\tau_3(\alpha_{1},\alpha_2,\alpha_6),\tau_3(-\alpha_{1},-\alpha_2,\alpha'_6),
\tau_4\left(\frac{\alpha'_6(1 - \alpha_{1})-\alpha_{6}(1 + \alpha_{1})}{\alpha_{2}},-\alpha_{6}-\alpha'_6\right)\right\}
\]
is isomorphic to  $\z^{2}_{2}$.
\item We suppose that $ \tau_4(\alpha_5,\alpha_6) \in \Gamma$. If $ \Gamma $ is not isomorphic to $ \z_{2} $,  then $\Gamma$ is one of the groups $\Gamma_{7}, \Gamma_{8}$.
\end{itemize}
\begin{teo}
Any finite abelian subgroup $\Gamma$ of  $Aut (\mathfrak h_3)$ isomorphic to $(\z_2)^k$ is one of the following
\begin{enumerate}
\item $k=1$, $\Gamma=\Gamma_1(\alpha_3,\alpha_6), \quad
\Gamma_2(\alpha_3,\alpha_5), \quad
\Gamma_3(\alpha_1,\alpha_2,\alpha_6),  \ \alpha_2 \neq 0, \quad
\Gamma_4(\alpha_5,\alpha_6),$
\item $k=2$, $\Gamma=\Gamma_{7}(\alpha_3,\alpha_5,\alpha_6), \quad \Gamma_8(\alpha_1,\alpha_2,\alpha_6,\alpha'_6).$
    \end{enumerate}
\end{teo}

\medskip

Let us assume now that $\Gamma$ is isomorphic to $\z_3^{k_3}$ with $k_3 \geq 2$. If $\tau \in \Gamma_5$, its matricial representation is
\[
\begin{pmatrix}
\displaystyle\frac{-1 - \sqrt{-3 - 4 \alpha_{2} \alpha_{3}}}{2}&\alpha_2&0\\
\alpha_3&\displaystyle\frac{-1 + \displaystyle\sqrt{-3 - 4 \alpha_{2} \alpha_{3}}}{2}&0\\
\alpha_5&\alpha_6&1
\end{pmatrix}.
\]
To simplify, we put $\lambda=\displaystyle\frac{-1 - \sqrt{-3 - 4 \alpha_{2} \alpha_{3}}}{2}$.  The eigenvalues of $\tau$ are $1,j,j^2$ and the corresponding eigenvectors $X_3,V,\overline{V}$ with
\[
V=(1,-\displaystyle\frac{\lambda-j}{\alpha_2},-\displaystyle\frac{\alpha_5}{1-j}+
\frac{\alpha_6(\lambda-j)}{\alpha_2(1-j)})
\]
if $\alpha_2 \neq 0$. If $\tau'$ is an automorphism of order $3$ commuting with $\tau$, then
\[
\tau' V=jV \quad {\mbox or} \quad  j^2V.
\]
But the two first components of $\tau'(V)$ are
\[
 \lambda' - \displaystyle\frac{\beta_2}{\alpha_2}(\lambda-j), \beta_3-\frac{\lambda'(\lambda-j)}{\alpha_2}
\]
where $\beta_i$ and $\lambda'$ are the corresponding coefficients of the matrix of $\tau'$. This implies
\[
\alpha_2\lambda'-\beta_2(\lambda-j)=\alpha_2j \ {\mbox{\rm or}} \ \alpha_2j^2.
\]
Considering the real and complex parts of this equation, we obtain
\[
\left\{
\begin{array}{l}
\alpha_2\lambda'-\beta_2\lambda=0, \\
\beta_2j=\alpha_2j \quad \mbox{or} \quad \alpha_2j^2.
\end{array}
\right.
\]
As $\alpha_2 \neq 0$, we obtain $\alpha_2=\beta_2$ and $\lambda=\lambda'$. Let us compare the second component of $\tau'(V)$.We obtain
\[
\beta_3\alpha_2-\lambda'(\lambda-j)=-(\lambda-j)j \quad \mbox{or} \quad  -(\lambda-j)j^2.
\]
As $\lambda=\lambda'$, we have in the first case $2\lambda j= j^2$ and in the second case $2\lambda j= j^3=1.$ In any case, this is impossible. Thus $\alpha_2=0$ and, from section 2.2, $\tau = Id$. This implies that $k_3=1$ or $0$.
\begin{teo}
Let $\Gamma$ be a finite abelian subgroup of $Aut(\h_3)$. Thus $\Gamma$ is isomorphic to one of the following
\begin{enumerate}
\item $\z_2 \times \z_2$,
\item $\z_2^{k_2} \times \z_3^{k_3} \times \cdots \times \z_p^{k_p}$ with $k_i=0 \quad \mbox{or} \quad  1$ for $i=2, \cdots, p$.
    \end{enumerate}
    \end{teo}
To prove the second part, we show identically to the case $i=3$ that $k_i=1$ as soon as $k_i \neq 0$.

\medskip

\noindent{\bf Example.} The group
\[
\Gamma_4(0,0) \times \Gamma_5(0,0,0,0) \times \cdots \times \Gamma_{6,k}(0,0,0,0)
\]
satisfies the second property of the theorem.

\noindent{\bf Remark.} 
We have determined the finite abelian subgroups of $Aut(\mathfrak h_3)$. There are non-abelian finite subgroups  with elements of order at most $3$. Take for example the subgroup generated by
\[
\sigma_{1}=\begin{pmatrix}
\-1&0&0\\
0&1&0\\
0&0&-1
\end{pmatrix},
\qquad
\sigma_{2} = \begin{pmatrix}
-\frac{1}{2}&\alpha&0\\
-\frac{3}{4\alpha}&-\frac{1}{2}&0\\
0&0&1
\end{pmatrix}  \qquad \alpha \neq 0
\]
The relations on the generators are
\[
\begin{cases}
\sigma^{2}_{1} = Id,\\
\sigma^{3}_{2} = Id,\\
\sigma_{1}\sigma_{2}\sigma_{1} = \sigma^{2}_{2}.
\end{cases}
\]
Thus the group generated by $\sigma_1$ and $\sigma_2$ is isomorphic to the symmetric group $\Sigma_3$ of degree  $3$.

\section{$\Gamma$-grading of $\mathfrak h_{3}$}

\subsection{Description of the $\z_2$ and $\z_2^2$-gradings}
Let $\Gamma$ be a finite abelian subgroup of $Aut(\h_3)$. We consider a $\Gamma $-grading of $\mathfrak h_3$
\[
\mathfrak h_3= \underset{\gamma \in \Gamma}{\oplus}\; \mathfrak{h}_{3,\gamma}
\]
such that $\mathfrak h_ {3, e} = \{0 \}$  where $ e $ is the identity of $\Gamma$. In this case, the space $\Gamma $-symmetric associated with this grading is isomorphic to the Heisenberg group $\mathbb{H}_3$ and then $\mathbb{H}_3$ can be studied in terms of  $\Gamma $-symmetric spaces. In this section, we are particulary interested by the case $\Gamma = \z_2$ or $\Gamma=\z_2 \times \z_2.$

\medskip

$\bullet$ If $\Gamma = \z_2$ then the grading of  $\mathfrak h_{3}$  is of the type
\[
\mathfrak h_{3} = \mathfrak g_{0} \oplus \mathfrak g_{1}
\]
with $\mathfrak g_{0} \neq \{0\}$. In this case the corresponding symmetric homogeneous space is isomorphic to  $\mathbb{H}_3/H$ where $H$ is a non trivial Lie subgroup of  $\mathbb{H}_3$ whose Lie algebra is $\g_0$.
The group $\mathbb{H}_3$ is not provided with a symmetric space structure.

\medskip

$\bullet$ If  $\Gamma = \z^{2}_2$ then  $\Gamma = \Gamma_7$ or  $\Gamma = \Gamma_8$. Recall that
\[
\Gamma_{7} = \left\{Id,\tau_1(\alpha_3,\alpha_6),\tau_2(-\alpha_3,\alpha_5),\tau_{4}\left(-\frac{\alpha_{3}\alpha_{6}}{2}-\alpha_{5},-\alpha_{6}\right)\right\}
\]
Denote by $L(V_1,\cdots,V_k)$ the real vector space generated by the vectors $V_1,\cdots,V_k$. Recall that each vector of the Heisenberg algebra is decomposed in the basis $\{X_1,X_2,X_3\}$.
The eigenspaces associated with $\tau_1(\alpha_3,\alpha_6)$ are
\begin{eqnarray*}
V_{1} &=& L\left((0,1,\frac{\alpha_{6}}{2})\right) \\
V_{-1} &=&  L\left((1,- \frac{\alpha_{3}}{2},0) ,(0,0,1)\right)
\end{eqnarray*}
The eigenspaces associated to  $\tau_2(-\alpha_3,\alpha_5)$ are
\begin{eqnarray*}
W_{1} &=& L\left((1,-\frac{\alpha_{3}}{2},\frac{\alpha_{5}}{2})\right) \\
W_{-1} &=&  L\left((0,1,0),(0,0,1)\right).
\end{eqnarray*}
Since $\tau_{4} = \tau_{1}\circ\tau_{2}$, the grading of $\mathfrak h_{3}$ associated with $\Gamma_{7}$ is
\begin{eqnarray*}
\mathfrak h_{3} &=&V_{1}\cap W_{1}\oplus V_{1}\cap W_{-1}\oplus V_{-1}\cap W_{1}\oplus V_{-1}\cap W_{-1} \\
&=& \{0\}\oplus \r\{(0,1,\frac{\alpha_{6}}{2})\}\oplus \r\{(1,-\frac{\alpha_{3}}{2},\frac{\alpha_{5}}{2})\}\oplus \r\{(0,0,1)\}.
\end{eqnarray*}
\medskip

Now consider the case where $\Gamma = \Gamma_8$
\[
\Gamma_{8} =\left\{Id,\tau_3(\alpha_{1},\alpha_2,\alpha_6),\tau_3(\alpha'_{1},\alpha'_2,\alpha'_6),
\tau_4\left(\frac{\alpha'_6(1 - \alpha_{1})-\alpha_{6}(1 + \alpha_{1})}{\alpha_{2}},-\alpha_{6}-\alpha'_6\right)\right\}
\]
The eigenspaces associated with  $\tau_3^1$ are
\begin{eqnarray*}
V_{1} &=& L\left((1,\frac{1-\alpha_{1}}{\alpha_{2}},\frac{\alpha_{6}}{\alpha_{2}})\right) \\
V_{-1} &=&  L\left((1,- \frac{1+\alpha_{1}}{\alpha_{2}},0) ,(0,0,1)\right)
\end{eqnarray*}
The eigenspaces associated with  $\tau_3^2$ are
\begin{eqnarray*}
W_{1} &=& L\left((1,-\frac{1+\alpha_{1}}{\alpha_{2}},\frac{\alpha'_{6}}{\alpha_{2}})\right) \\
W_{-1} &=&  L\left((1,\frac{1-\alpha_{1}}{\alpha_{2}},0),(0,0,1)\right).
\end{eqnarray*}
The grading associated with $\Gamma_8$ is therefore
\begin{eqnarray*}
\mathfrak h_{3} &=&V_{1}\cap W_{1}\oplus V_{1}\cap W_{-1}\oplus V_{-1}\cap W_{1}\oplus V_{-1}\cap W_{-1} \\
&=& \{0\}\oplus \r\{(1,\frac{1-\alpha_{1}}{\alpha_{2}},0)\}\oplus \r\{(1,- \frac{1+\alpha_{1}}{\alpha_{2}},0)\}\oplus \r\{(0,0,1)\}.
\end{eqnarray*}

\begin{prop}
The $\z^{2}_{2}$-grading of $\mathfrak h_{3}$ correspond to one of the following:
\begin{eqnarray*}
\mathfrak h_{3} &=&
\{0\}\oplus \r\{(X_{2}+\frac{\alpha_{6}}{2}X_{3}\}\oplus
 \r\{(X_{1}- \frac{\alpha_{3}}{2}X_{2}+\frac{\alpha_{5}}{2}X_{3})\}\oplus \r\{X_{3}\}\\
 & &\\
\mathfrak h_{3}
&=& \{0\}\oplus \r\{(X_{1}+\frac{1-\alpha_{1}}{\alpha_{2}}X_{2}\}\oplus
 \r\{(X_{1}- \frac{1+\alpha_{1}}{\alpha_{2}}X_{2})\}\oplus \r\{X_{3}\}
\end{eqnarray*}
\end{prop}

\medskip

\noindent{\bf Remark. } If $\Gamma = \z_{3}$, we consider the complexification $\mathfrak h_ {3, \C}$ of the Heisenberg  algebra. We still denote by $X_1, X_2, X_{3}$ the complex basis of $\mathfrak h_{3, \C}$ corresponding to the given basis of $\mathfrak h_3$. The grading in this case is defined by the complex eigenspaces of $\tau_5$. They are
\begin{eqnarray*}
V_{1} &=& \C\{(0,0,1)\}\\
V_{j}&=& \C \{(1,\frac{1 + 2 j + \sqrt{-3 - 4 \alpha_{2}\alpha_{3}}}{2\alpha_{2}},0)\}\\
V_{\bar{j}}&=&\C \{(1,\frac{1 + 2 \bar{j} + \sqrt{-3 - 4 \alpha_{2}\alpha_{3}}}{2\alpha_{2}},0)\}
\end{eqnarray*}
We have the grading
\[
\mathfrak h_{3,\C} = V_{1}\oplus  V_{j}\oplus V_{\bar{j}}
\]

\subsection{Classification of  $\z_2^2$-grading up an automorphism}
\begin{lem}
There is an automorphism $\sigma \in Aut(\mathfrak h_{3})$ such that
\[
\sigma^{-1} \Gamma_{7} \sigma = \Gamma_{8}
\]
\end{lem}
\begin{proof}
Denote by $(\alpha_{3},\alpha_{5},\alpha_{6})$ the parameters of the family $\Gamma_{7}$
and by  $(\alpha_{1},\alpha_{2},\alpha'_{6},\alpha''_{6})$ those of $\Gamma_{8}$. If $\alpha_{1}^{2} \neq 1$, then the automorphism
\[
\sigma =
\begin{pmatrix}
\gamma&\displaystyle\frac{\gamma \alpha_{2}}{ \alpha_{1}-1}&0\\
\delta&-\displaystyle\frac{\alpha_{2} (\gamma \alpha_{3}+\delta-\alpha_{1} \delta)}{-1+ \alpha_{1}^2}&0\\
\rho&
\mu&-\displaystyle\frac{\gamma \alpha_{2} (\gamma  \alpha_{3}+2 \delta)}{-1+ \alpha_{1}^2}
\end{pmatrix}
\]
with
\begin{eqnarray*}
\rho &=&\frac{
(2 \gamma \alpha_5  + \gamma \alpha_3 \alpha_6  + 2 \alpha_6 \delta )}{4}\\
&+&
\frac{(2 \gamma^2 \alpha_3 a'_6 + 4 \gamma \delta a'_6) (1 + \alpha_1) + (2 \gamma^2 \alpha_3
a''_6  + 4 \gamma \delta a''_6)(\alpha_1-1)}{4(\alpha_{1}^2 -1)}
\\
& &
\\
\mu&=& \frac{2  \gamma \alpha_2 \alpha_5 (1 + \alpha_1) +   \alpha_2 \alpha_6 (\gamma  \alpha_3 +
 2\delta) (\alpha_1-1) + (2  \gamma^2 \alpha_2 \alpha_3 + 4  \gamma \alpha_2 \delta) ( a'_6 + a''_6)}{4(\alpha_{1}^2 -1)}
 \end{eqnarray*}
answers to the question. \\
If $\alpha_{1} = 1$, we consider
\[
\sigma =
\begin{pmatrix}
0&\beta&0\\
\gamma&\displaystyle\frac{-\beta \alpha_{3} + \alpha_{2} \gamma}{2}&0\\
\gamma \left(\displaystyle\frac{\alpha_{6}}{2}+\displaystyle\frac{\beta \alpha'_{6}}{\alpha_{2}}\right)&
\displaystyle\frac{\alpha_{2} \gamma \alpha_{6}+2 \beta (\alpha_{5}+\gamma \alpha'_{6}+\gamma \alpha''_{6})}{4}&
-\beta \gamma
\end{pmatrix}
\]
and 
if $\alpha_{1} = -1$, we take
\[
\sigma =
\begin{pmatrix}
-\displaystyle\frac{2 \beta}{\alpha_{2}}&\beta &0\\
\displaystyle\frac{\beta \alpha_{3}}{\alpha_{2}}&\delta&0\\
\frac{-\alpha_{2} \beta \alpha_{5}-(\beta^2 \alpha_{3} +)\alpha''_{6}}{\alpha_{2}^2}&
\frac{2 \alpha_{2} \beta (\alpha_{5}+ \alpha_{3} \alpha_{6})+
(2 \beta^2 \alpha_{3} + 4\beta \delta )(\alpha'_{6}+\alpha''_{6})+
2 \alpha_{2} \delta \alpha_{6}}{4 \alpha_{2}}&
-\displaystyle\frac{\beta ^2 \alpha_{3} +2 \beta \delta}{\alpha_{2}}
\end{pmatrix}
\]
These automorphisms give an equivalence between the two subgroups.
\end{proof}
\medskip

\noindent {\bf Consequence} Let $\mathfrak h_3 = \{0\}\oplus \mathfrak h_{3,a_1}\oplus \mathfrak h_{3,a_2}  \oplus \mathfrak h_{3,a_3} = \{0\}\oplus \mathfrak h'_{3,a_1}\oplus \mathfrak h'_{3,a_2}  \oplus \mathfrak h'_{3,a_3}$ be the two  $\z^2_2$-gradings de $\mathfrak h_3$, where $\{0,a_1,a_2,a_3\}$ are the elements of  $\z^2_2$.
There exists $\sigma\in Aut(\mathfrak h_3)$ such that
\[
\mathfrak h'_{3,a_i} = \sigma (\mathfrak h_{3,a_i})
\]
Thus, these grading are equivalent. (The equivalence of two grading is defined in \cite{[B.G]}).

\begin{lem}
There exists $\sigma\in Aut(\mathfrak h_3)$ such that 
\[
\left\{
\begin{array}{l}
 \sigma^{-1} \tau_{1}(\alpha_{3},\alpha_{6}) \sigma = \tau_{1}(0,0), \\
 \sigma^{-1} \tau_{2}(-\alpha_{3},\alpha_{5}) \sigma = \tau_{2}(0,0).
 \end {array}
 \right.
 \]
\end{lem}
\begin{proof} Indeed if
\[
\sigma =\begin{pmatrix}
1&0&0\\
-\frac{\alpha_{3}}{2}&1&0\\
\rho&\frac{\alpha_{6}}{2}&1
\end{pmatrix}
\]
then
\[
\sigma^{-1} =\begin{pmatrix}
1&0&0\\
\frac{\alpha_{3}}{2}&1&0\\
-\frac{\alpha_{3} \alpha_{6}}{4}- \rho&-\frac{\alpha_{6}}{2}&1
\end{pmatrix}
\]
and
\[
\sigma^{-1} \tau_{1}(\alpha_{3},\alpha_{6}) \sigma = \tau_{1}(0,0)
\]
This automorphism satisfies
\[
\sigma^{-1} \tau_{2}(-\alpha_{3},\alpha_{5}) \sigma = \tau_{2}(0,\alpha_{5} - 2 \rho)
\]
If $\rho = \frac{\alpha_{5}}{2}$ that is
\[
\sigma =\begin{pmatrix}
1&0&0\\
-\frac{\alpha_{3}}{2}&1&0\\
 \frac{\alpha_{5}}{2}&\frac{\alpha_{6}}{2}&1
\end{pmatrix}
\]
then we have
\[
\sigma^{-1} \tau_{2}(-\alpha_{3},\alpha_{5}) \sigma = \tau_{2}(0,0) .
\]
\end{proof}
\noindent From the previous Lemma we have
\begin{prop}
Every $\z^{2}_{2}$-grading on $\mathfrak h_{3}$ is equivalent to the grading defined by
\[
\Gamma_{7}(0,0,0) = \{ Id,\tau_{1}(0,0),\tau_{2}(0,0),\tau_{4}(0,0)\}.
\]
\end{prop}
\noindent This grading corresponds to
\[
\mathfrak h_{3} =\{0\} \oplus \r(X_{2})\oplus \r(X_{2}) \oplus \r(X_{3}).
\]

\section{Riemannian structures $\z^{2}_{2}$-symmetric}

Let $ G / H$ be an homogeneous  $\Gamma $-symmetric space. We denoted by $\rho:\Gamma\rightarrow Aut (G) $ the injective homomorphism of groups. Each element $\rho (\gamma) $ for $\gamma\in\Gamma $ is called a symmetry of  the $\Gamma $-symmetric space.
\begin{defi}
The $\Gamma $-symmetric homogeneous space  $ G / H $ is called Riemannian $\Gamma $-symmetric  if there exists on $ G / H $ a Riemannian metric $ g $ such that
\begin{enumerate}
\item $g$ is $ G $-invariant,
 \item the symmetries $\rho (\gamma) $, $\gamma\in\Gamma $, are isometries.
 \end{enumerate}
\end{defi}
According to \cite {[G.R]}, such a metric is completely determined by a bilinear form $ B $ on the Lie algebra $\mathfrak g $ such that  
\begin{enumerate}
\item $B$ is $ ad\mathfrak h $ invariant ($\mathfrak h =\mathfrak {g} _e) $
\item
$B(\mathfrak g_{\gamma},\mathfrak g_{\gamma'}) = 0$ if  $\gamma \neq \gamma' \neq e$
\item The restriction of $ B $ to $\oplus_ {\gamma\neq e}\mathfrak g_\gamma $ is  positive definite.
\end{enumerate}
Consider on $\mathbb{H}_3$, the Heisenberg group, a  $\z^2_2 $-symmetric structure. It is determined, up to equivalence, by the $\z^2_2 $-grading of $\mathfrak h_3$
\[
\mathfrak h_{3} =\{0\} \oplus \r(X_{1})\oplus \r(X_{2}) \oplus \r(X_{3})
\]
Since every automorphism of $\mathfrak h_3 $ is an isometry of any invariant  Riemannian metric on $\H_3 $, we deduce
\begin{teo}
Any Riemannian structure $\z^2_2 $-symmetric over $\H_3 $ is isometric to the Riemannian structure associated with the grading
\[
\mathfrak h_{3} =\{0\} \oplus \r(X_{1})\oplus \r(X_{2}) \oplus \r(X_{3})
\]
and the Riemannian metric is written
\[
g = \omega^{2}_{1} + \omega^{2}_{2} + \lambda^{2} \omega^{2}_{3}
\]
with $\lambda \neq 0$, where $\{\omega_1,\omega_2,\omega_3\}$ is the dual basis of $\{X_1,X_2,X_3\}$.
\end{teo}
\begin{proof}
Indeed, as the components of the grading are orthogonal, the Riemannian metric $ g $, which coincides with the form $ B $ verifies
\[
g=\alpha_{1} \omega^{2}_{1} + \alpha_{2} \omega^{2}_{2} + \alpha_{3} \omega^{2}_{3}
\]
with $\alpha_{1} > 0,\alpha_{2}> 0,\alpha_{3}> 0$.
According to \cite{[G.P1]}, we reduce the coefficients to  $\alpha_1 =\alpha_2 =1$.
\end{proof}
\medskip

\noindent{\bf Remark.} According to \cite{[H]} and \cite{[G.P2]}, this metric is naturally reductive for any $\lambda$.

\section{Lorentzian  $\z^{2}_{2} $ - symmetric structures on $\H_{3}$}
We say that a homogeneous space $(M = G/H, g)$ is {\it Lorentzian} if the canonical action of $G$ on $M$ preserves a Lorentzian metric (i.e. a smooth field  of non degenerate quadratic form of signature $(n-1,1)$).
\begin{prop}[\cite{[D.Z]}]
Modulo an automorphism and a multiplicative constant, there exists on $\h_{3}$ one left-invariant metric assigning a strictly positive length on the center of
$\mathfrak h_{3} $.
\end{prop}
The Lie algebra $\mathfrak h_ {3} $ is generated by the central vector $X_3$ and  $X_1$ and $X_2 $ such that $ [X_1, X_2] =  X_3$. The automorphisms of the Lie algebra preserve the center and then send the element $X_3$ on $\lambda X_3$, with $\lambda \in \r ^{*}$. Such an automorphism acts on the plane generated by $X_1$ and $X_2 $ as an automorphism of determinant $\lambda $. \\
It is shown in \cite {[R]} and \cite {[R.R]} that, modulo an automorphism of $\mathfrak h_ {3} $, there are three classes of invariant Lorentzian metrics on $\H_{3}$,  corresponding to the cases where $||X_3||$ is negative, positive or zero.\\
We propose to look at the Lorentz metrics that are associated with  the  $\z_2^2$-symmetric structure over $\H_3$.
\begin{defi}
Let $ M = G / H$ be a homogeneous   $\Gamma $-symmetric space. Let $g $  be a Lorentz metric on $ M $.
We say that the metric $ g $ is $\z^2_2 $-symmetric Lorentzian if one of the two conditions is satisfied:
\begin{enumerate}
\item The homogeneous non trivial components $\g_\gamma$ of the $\Gamma$-graded Lie algebra of $G$ are orthogonal and non-degenerate with respect to the induced bilinear form $ B $.
\item One non trivial component $\g_{\lambda_0}$ is degenerate, the other components are orthogonal and non-degenerate,  and the exists a component $\g_{\lambda_1}$ such that the signature of the restriction to $B$ at $\g_{\lambda_0} \oplus \g_{\lambda_1}$ is  $(1,1)$.
\end{enumerate}
\end{defi}
If  $\mathfrak g$  is the Heisenberg algebra equipped with a $\z^2_2 $-grading, then by automorphism, we can  reduce to the case where
$\Gamma = \Gamma_7$. In this case, the grading of $\mathfrak  h_{3} $ is given by:
\[
\mathfrak h_{3}= \mathfrak g_{0} + \mathfrak g_{+ -} + \mathfrak g_{- +} +\mathfrak g_{- -}
\]
with
\[
\begin{cases}
\mathfrak g_{0} = \{0\},\\
\mathfrak g_{+ -} =  \r\left(X_{2} -\frac{\alpha_{6}}{2}X_{3}\right),\\
\mathfrak g_{- +} =  \r\left(X_{1} -\frac{\alpha_{3}}{2}X_{2}+ \frac{\alpha_{5}}{2}X_{3}\right),\\
\mathfrak g_{- -}  = \r\left(X_{3}\right).
\end{cases}
\]
Assume
\[
Y_{1} = X_{1} -\frac{\alpha_{3}}{2}X_{2}+ \frac{\alpha_{5}}{2}X_{3}\qquad
Y_{2} = X_{2} -\frac{\alpha_{6}}{2}X_{3}\qquad 
Y_{3} = X_{3}.
\]
The dual basis is
\[
\vartheta_{1} = \omega_{1} \qquad \vartheta_{2} = \omega_{2} + \frac{\alpha_{3}}{2} \omega_{1} \qquad \vartheta_{3} = \omega_{3} -\frac{\alpha_{6}}{2}\omega_{2} -\left(\frac{\alpha_{3}\alpha_{6}}{4} + \frac{\alpha_{5}}{2}\right)\omega_{1}
\]
where $\{\omega_1,\omega_2,\omega_3\}$ is the dual basis of the base $\{X_1,X_2,X_3\}$.\\
\begin{description}
\item[Case $I$] The components $ \mathfrak g_{+ -}, \mathfrak g_{- +},\mathfrak g_{- -}$ are non-degenerate. The quadratic form induced on $\mathfrak  h_3$ therefore writes
\[
g = \lambda_{1}\omega^{2}_{1} +
\lambda_{2}\left(\omega_{2} + \frac{\alpha_{3}}{2} \omega_{1}\right)^{2} +
\lambda_{3} \left(\omega_{3} - \frac{\alpha_{6}}{2} \omega_{2} -
\left(\frac{\alpha_{5}}{2} + \frac{\alpha_{3} \alpha_{6}}{4}\right)\omega_{1}\right)^{2}
\]
with $\lambda_{1},\lambda_{2},\lambda_{3} \neq 0$.
The change of basis associated with the matrix
\[
\begin{pmatrix}
1&0&\\
\frac{\alpha_{3}}{2}&1&0\\
-\frac{\alpha_{5}}{2} - \frac{\alpha_{3} \alpha_{6}}{4}&- \frac{\alpha_{6}}{2}&1
\end{pmatrix}
\]
is an automorphism. Thus $ g $ is isometric to
\[
g = \lambda_{1}\omega^{2}_{1} +
\lambda_{2}\omega^{2}_{2} + \lambda_{3}\omega^{2}_{3} .
\]
Since the signature is $(2,1) $ one of the $\lambda_ i$ is negative and the two others positive.

\begin{prop}
Every Lorentzian metric $\z^2_2 $-symmetric $ g $ on $\H_3$ such that the components of the grading of $\mathfrak {h}_3 $ are non degenerate, is 
reduced to one of these two forms:
\[
\begin{cases}
g = - \omega^{2}_1 + \omega^{2}_2 + \lambda^{2} \omega^{2}_3\\
g = \omega^{2}_1 + \omega^{2}_2  - \lambda^{2} \omega^{2}_3
\end{cases}
\]
\end{prop}

\item[Case $II$] Suppose that a component is degenerate.
When this component is $\r (x_2 +\frac {\alpha_6}{2} X_3) $ or $\r (X_1 -\frac {\alpha_3} {2} X_2 +\frac{\alpha_5}{2}  X_3 ) $ then, by automorphism, we reduce to the above case.\\
Suppose then that the component containing the center is degenerate.\\
Thus the quadratic form induced on $\mathfrak h_3$ is written
\begin{eqnarray*}
g &=& \omega^{2}_{1} +
\left[\omega_{3} - \frac{\alpha_{6}}{2} \omega_{2} -
\left(\frac{\alpha_{5}}{2} + \frac{\alpha_{3} \alpha_{6}}{4}\right)\omega_{1}\right]^{2}\\
&-& \left[\omega_{2} - \omega_{3} + \frac{\alpha_{6}}{2} \omega_{2} +
\left(\frac{\alpha_{5}}{2} + \frac{\alpha_{3} \alpha_{6}}{4}\right)\omega_{1}\right]^{2}.
\end{eqnarray*}

The change of basis associated with the matrix
\[
\begin{pmatrix}
1&0&\\
\frac{\alpha_{3}}{2}&1&0\\
-\frac{\alpha_{5}}{2} - \frac{\alpha_{3} \alpha_{6}}{4}&- \frac{\alpha_{6}}{2}&1
\end{pmatrix}
\]
is given by an automorphism. Thus $g$ is isomorphic to
\[
g = \omega^{2}_{1} +  \omega^{2}_{3} -
(\omega_{2} - \omega_{3})^{2}.
\]
\begin{prop}
Every Lorentzian  $\z^2_2 $-symmetric $ g $ metric on $\H_3$ such that the component of the grading of $\mathfrak {h} _3 $ containing the center is degenerate, is reduced to the form
\[ g = \omega^{2}_{1} +  \omega^{2}_{3} -
(\omega_{2} - \omega_{3})^{2}.
\]
\end{prop}
From \cite{[C.P]} is the only flat Lorentzian metric, left invariant on the Heisenberg group.
\end{description}

\strut\hfill  Universit\'e de Haute Alsace,\\
\strut\hfill LMIA\\
\strut\hfill 4 rue des fr\`eres Lumi\`ere,
68093 Mulhouse, France \\
\vspace{2mm}
\strut\hfill Michel.Goze@uha.fr \\

\strut\hfill Universit\`a degli Studi di Cagliari,\\
\strut\hfill Dipartimento di Matematica e Informatica\,\\
\strut\hfill Via Ospedale 72, 09124 Cagliari, ITALIA\\
\vspace{2mm}
\strut\hfill piu@unica.it
\bigskip

\end{document}